\documentclass[runningheads]{llncs}
\usepackage[T1]{fontenc}
\usepackage{graphicx}
\usepackage{amsmath}
\usepackage{amssymb}
\usepackage{amsfonts}
\usepackage{newtxtext, newtxmath}
\usepackage{hyperref}
\usepackage{epsfig}
\usepackage{graphicx}
\usepackage{subfigure}
\usepackage[utf8]{inputenc}
\usepackage[english]{babel}
\usepackage{CJKutf8}
\usepackage{url}
\usepackage{xcolor}
\usepackage{tikz, tikz-3dplot}
\usepackage{tikz-qtree}

\usepackage{amsthm}
\usepackage{algorithm}
\usepackage{algorithmicx}
\usepackage[noend]{algpseudocode}
\usetikzlibrary{positioning,calc}
\usepackage{etoolbox}
\usetikzlibrary{arrows.meta, positioning, calc}
\usetikzlibrary{decorations.markings}
\usetikzlibrary{decorations.text}

\algdef{SE}[DOWHILE]{Do}{doWhile}{\algorithmicdo}[1]{\algorithmicwhile\ #1}%

\newcommand{\rounddown}[1]{\left\lfloor#1\right\rfloor}

\newcommand{\Z}{\mathbb{Z}}

\begin{document}
\title{A Min-Max Relation on Dicuts and Dijoins in Weighted Chordal Digraphs}
%

\author{Gérard Cornuéjols\orcidID{0000-0002-3976-1021} \and
Siyue Liu\orcidID{0000-0002-3553-9431} \and \\
R. Ravi\orcidID{0000-0001-7603-1207}}
\authorrunning{G. Cornuéjols et al.}

\institute{Carnegie Mellon University, Pittsburgh, USA\\
\email{\{gc0v,siyueliu,ravi\}@andrew.cmu.edu}}

\maketitle              
\begin{abstract}
In a digraph, a dicut is a cut where all the arcs cross in one direction. A dijoin is a subset of arcs that intersects every dicut. Edmonds and Giles conjectured that in a weighted digraph, the minimum weight of a dicut is equal to the maximum size of a packing of dijoins. This has been disproved. However, the unweighted version conjectured by Woodall remains open. We prove that the Edmonds-Giles conjecture is true if the underlying undirected graph is chordal. We also give a strongly polynomial-time algorithm to construct such a packing.

\end{abstract}

\section{Introduction}
In a digraph $D=(V,A)$, for a proper node subset $U\subsetneqq V, U\neq\emptyset$, denote by $\delta^+(U)$ and $\delta^-(U)$ the arcs leaving and entering $U$, respectively. Let $\delta(U):=\delta^+(U)\cup \delta^-(U)$. A \textit{dicut} is a set of arcs of the form $\delta^+(U)$ such that $\delta^-(U)=\emptyset$. A \textit{dijoin} is a set of arcs that intersects every dicut at least once. 
If $D$ is a weighted digraph with arc weights $w:A\rightarrow \mathbb{Z}_+$, we say that $(D,w)$ can \textit{pack} $k$ dijoins if there exist $k$ dijoins $J_1,...,J_k$ such that no arc $e$ is contained in more than $w(e)$ of $J_1,...,J_k$.
In this case, $J_1,...,J_k$ is a \textit{packing} of dijoins in $(D,w)$. 
Assume that the minimum weight of a dicut is $\tau$. Clearly, $\tau$ is an upper bound on the maximum number of dijoins in a packing. 
Edmonds and Giles \cite{edmonds1977min} conjectured that the other direction also holds true:
\begin{conjecture}[Edmonds-Giles]\label{conj:E-G}
    In a weighted digraph, the minimum weight of a dicut is equal to the maximum size of a packing of dijoins.
\end{conjecture}
Schrijver \cite{schrijver1980counterexample} disproved the above conjecture by providing a counterexample (see Figure \ref{fig:Younger}). However,
the unweighted version of the Edmonds-Giles conjecture, proposed by Woodall \cite{woodall1978menger}, is still open.
Woodall's conjecture asserts that the minimum size of a dicut equals the maximum number of disjoint dijoins.

There has been great interest in understanding when the Edmonds-Giles conjecture is true (e.g., \cite{hwang2022edmonds,williams2004packing,abdi2023packing,abdi2024arc,shepherd2005visualizing,chudnovsky2016disjoint,schrijverobervation}), which could be an important stepping stone towards resolving Woodall's conjecture. Despite many efforts, it is still not known whether there exists a packing of $2$ dijoins in a weighted digraph when $\tau$ is large enough \cite{shepherd2005visualizing}, in contrast to the existence of a packing of $\rounddown{\frac{\tau}{6}}$ dijoins in unweighted digraphs \cite{cornuejols2025approximately}. Abdi et al. \cite{abdi2025strong} recently showed that the Edmonds-Giles conjecture is true when $\tau=2$ and the digraph induced by the arcs of weight $w(e) \geq 1$ is (weakly) connected, settling a conjecture of Chudnovsky et al. \cite{chudnovsky2016disjoint}. Abdi et al. \cite{abdi2023packing} proved the Edmonds-Giles conjecture to be true when some parameter characterizing the ``discrepancy" of a digraph takes several special values. Among the special cases where the Edmonds-Giles conjecture was proved to be true, probably the most well-known one is \textit{source-sink connected} digraphs, where there is a directed path going from every source to every sink. This result was proved independently by Schrijver \cite{schrijver1982min}, and Feofiloff and Younger \cite{feofiloff1987directed}. Lee and Wakabayashi \cite{lee2001note} proved that the conjecture also holds true for digraphs whose underlying graphs are series-parallel.

\begin{figure}[htbp]
	\centering
	\includegraphics[scale=0.3]{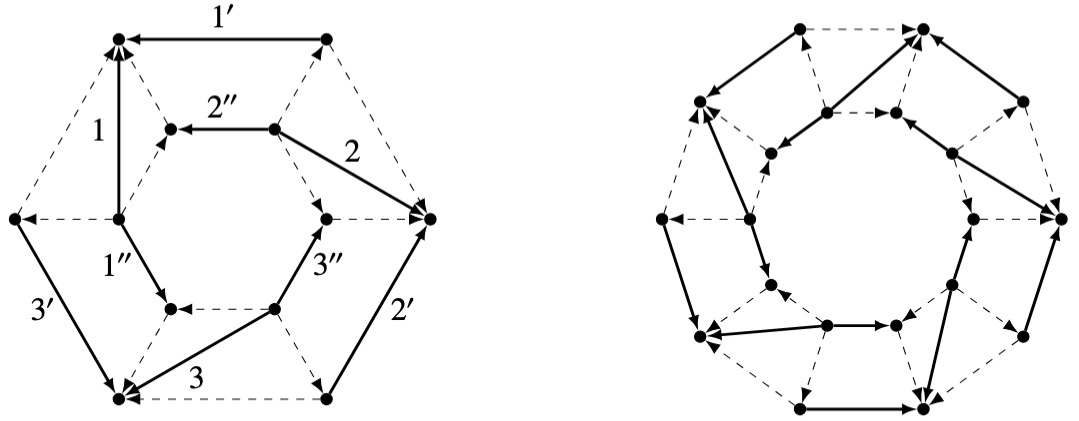}
 \caption{Left: Schrijver's counterexample to the Edmonds-Giles conjecture. The solid arcs have weight $1$ and the dashed arcs have weight $0$. The minimum weight of a dicut equals $2$. Among all minimal dicuts, six have weight $2$ such as $D_1\supseteq \{1,1'\}$ and $D_2\supseteq \{1,1''\}$, and four have weight $3$ such as $D_3\supseteq \{1,2,3\}$ and $D_4\supseteq \{1',2'',3\}$. One can verify that the solid arcs do not contain $2$ disjoint dijoins. Right: a generalization of Schrijver's example due to Younger. It extends to an infinite family with a ring of length $4k+2$ and $2k+1$ solid paths for $k\in \Z_{\geq 1}$.}
  \label{fig:Younger}
\end{figure}

Younger extended Schrijver's counterexample to an infinite family \cite{williams2004packing} (see Figure \ref{fig:Younger}). 
Cornu\'ejols and Guenin gave two more counterexamples \cite{cornuejols2002dijoins}. 
Williams \cite{williams2004packing} constructed more infinite classes, but defined a notion of minimality that reduces these counterexamples to the earlier known classes.  We observe that all of the known counterexamples to the Edmonds-Giles conjecture have a large chordless cycle (actually of length at least $6$). This motivates us to prove the following result.

A graph is \textit{chordal} if there is no chordless cycle of length more than $3$.

\begin{theorem}\label{thm:main}
    The Edmonds-Giles conjecture is true for digraphs whose underlying undirected graph is chordal.
\end{theorem}
This result only depends on the underlying undirected graph, regardless of its orientation. Our proof constructs a packing of $\tau$ dijoins in strongly polynomial time.

Chordal graphs are perfect \cite{hajnal58}, and in fact they are closely associated with the birth of perfect graph theory \cite{berge1960}. In the literature, chordal graphs are also referred to as {\em triangulated graphs}, or {\em rigid-circuit graphs}.  Dirac \cite{dirac1961rigid} observed the following striking property of chordal graphs.

 In a graph, a vertex is called  \emph{simplicial} if its neighbors form a clique.
 
\begin{lemma}[\cite{dirac1961rigid}]\label{lemma:simplicial}
Every chordal graph has a simplicial vertex. Moreover,
after removing a simplicial vertex, the resulting graph remains chordal.
\end{lemma}
 We will use this result in our proof of Theorem~\ref{thm:main}. Lemma~\ref{lemma:simplicial} implies that one can recursively remove simplicial vertices. This process is called a \textit{perfect elimination scheme}. Rose \cite{rose1970triangulated} showed that a graph is chordal if and only if it has a perfect elimination scheme. Rose, Lueker and Tarjan \cite{roseluekertarjan1976} showed how to find a perfect elimination scheme efficiently and how to recognize chordal graphs in linear time. Chordal graphs have other elegant characterizations. For example, Buneman \cite{Buneman1974} showed that a graph is chordal if and only if it is the intersection graph of a family of subtrees of a tree.

\section{Proof of Theorem \ref{thm:main}}

Let $D=(V,A)$ be a digraph whose underlying undirected graph is chordal and let $w: A\rightarrow \Z_{\geq 0}$ be a weight function. Suppose the minimum weight of a dicut is $\tau$. 

    Note that, if $D$ contains a directed cycle $C$, the 
    arcs in $C$ do not belong to any dicut. Therefore no minimal dijoin of $D$ contains an arc of $C$. Furthermore, after contracting all the arcs of $C$, the underlying graph remains chordal. It follows that
    we can assume w.l.o.g. that $D$ has no directed cycle. 

\subsection{Removing a simplicial vertex}

    By Lemma \ref{lemma:simplicial}, there exists a simplicial vertex $v$, and after removing $v$ the underlying graph remains chordal. 
    
    Let us begin with some observations. First, the neighbor set of $v$, denoted by $N(v)$, induces an {\it acyclic tournament}, that is, there is a total ordering $v_1,...,v_k$ of the nodes in $N(v)$ such that $v_iv_j\in A$ if and only if $1\leq i<j \leq k$. This is because $N(v)$ is a clique and there is no directed cycle. Moreover, observe that if $vv_i\in A$, then we must have $vv_j\in A$ for every $i<j\leq k$. Otherwise, $v, v_i,v_j$ would form a directed cycle. Therefore, there is some $s\in\{0,1,...,k\}$ such that $v_iv\in A$ for $1\leq i\leq s$ and $vv_i\in A$  for $s+1\leq i\leq k$ (see Figure \ref{fig:simplicial}). When $s=0$, we have $vv_i \in A$ for all $i \in [k]$, i.e., $v$ is a source.  
When $s=k$, we have $v_i v \in A$ for all $i \in [k]$, i.e., $v$ is a sink.
\begin{figure}[htbp]
	\centering
	\includegraphics[scale=0.22]{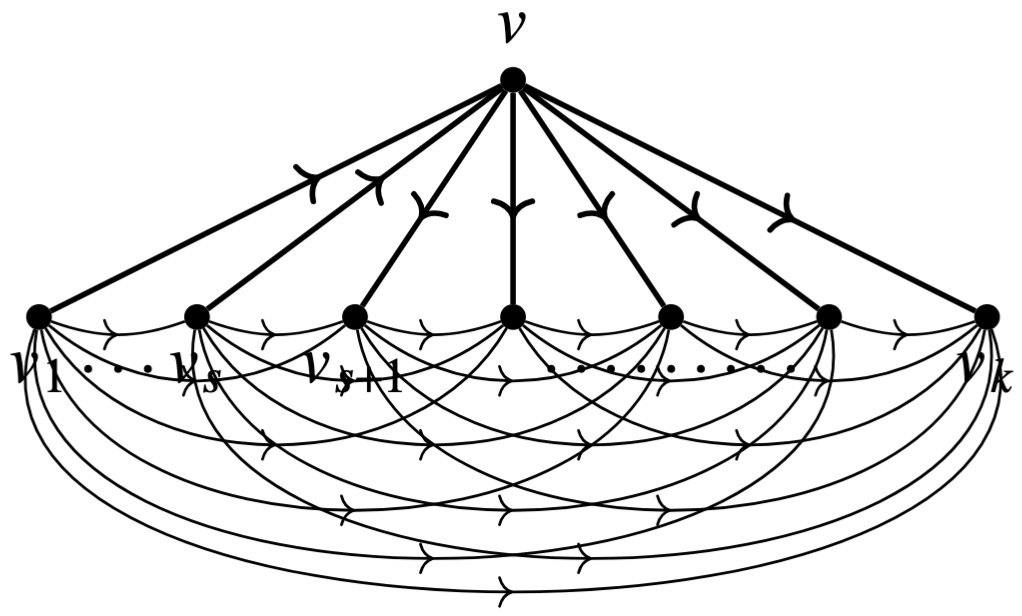}
  \caption{Orientations of the arcs adjacent to the simplicial vertex $v$ and of the arcs in the clique induced by its neighbors $N(v)=\{v_1,...,v_k\}$.}
  \label{fig:simplicial}
\end{figure}

     We prove Theorem~\ref{thm:main} by induction on $|V|$. Let $D'=(V\setminus\{v\},A')$ be the digraph obtained by deleting a simplicial vertex $v$. 

     Usually, after deleting arcs of a digraph, the new digraph has more dicuts than the original one because some non-dicuts may become dicuts. As a result, the minimum weight of a dicut can drop drastically. This is the general difficulty in proving Woodall's conjecture by induction. However, this is not the case for $D'$. In particular, we prove the following:
\begin{lemma}\label{lemma:dicut_preserve}
    Every dicut in $D'$ is a dicut in $D$ restricted to the arcs in $A'$.
\end{lemma}
\begin{proof}
    Given a dicut $\delta^+(U)$ in $D'$, notice that it intersects the directed path $v_1\rightarrow v_2\rightarrow\ldots\rightarrow v_k$ at most once. If it does not intersect the directed path, then either $\{v_1,...,v_k\}\subseteq U$ or $\{v_1,...,v_k\}\cap U=\emptyset$. Assume the former and the latter case follows by symmetry. Then, $\delta^+(U\cup\{v\})$ is a dicut in $D$, because $v$ is not connected to any vertex in $V\setminus U$. Otherwise, assume that $\delta^+(U)$ intersects the directed path on $v_rv_{r+1}$ for some $r\in\{1,...,k-1\}$. 
    We have $\{v_1,...,v_k\}\cap U=\{v_1,...,v_r\}$.
    If $r\leq s$, then $\delta^+(U)$ is a dicut in $D$. This is because, since $U\cap \{v_{s+1},...,v_k\}=\emptyset$, there is no arc going from $v$ to $U$. Similarly, if $r\geq s+1$, then $\delta^+(U\cup\{v\})$ is a dicut in $D$.
\end{proof}
      
\subsection{Transferring the weight}
    Although Lemma \ref{lemma:dicut_preserve} shows that deleting $v$ will not create new dicuts, the weight of each dicut can decrease. To counter this, we transfer the weight of arcs adjacent to $v$ to the clique $N(v)$ to make sure that the weight of each dicut does not decrease.

    Next, we give a construction of a weight function $w':A'\rightarrow \Z_{\geq 0}$ such that the minimum weight of a dicut in $A'$ under the new weight $w'$ is at least $\tau$. For every $e\in A'$ with at most one endpoint in $\{v_1,...,v_k\}$, let $w'(e)=w(e)$. Denote by $u_i$ the weight of the arc connecting $v$ and $v_i$, $\forall i\in[k]$. 
    
    \paragraph{\textbf{Case 1}}: $\delta(v)$ is not a dicut.
    
    We may assume w.l.o.g. that $\sum_{i=1}^s u_i\leq \sum_{i=s+1}^k u_i$ (otherwise we
    can flip the direction of all arcs in $D$ and consider the resulting digraph). Now, we assign an arbitrary weight $u'\leq u$ on $N(v)$ such that $\sum_{i=s+1}^k u'_i=\sum_{i=1}^s u'_i=\sum_{i=1}^s u_i$ and then construct a bipartite $u'$-matching between $\{v_1,...,v_s\}$ and $\{v_{s+1},...,v_k\}$. For simplicity, we construct $u'$ in the following way. Let $t$ be the largest index in $\{s+1,...,k\}$ satisfying $\sum_{i=t+1}^k u_i\leq \sum_{i=1}^s u_i\leq \sum_{i=t}^k u_i$. Define 
\begin{equation}\label{eq:w'-matching}
    \begin{aligned}
        u'_i=\begin{cases}
        u_i, & i\in\{1,...,s\}\cup\{t+1,...,k\};\\
        \sum_{i=1}^s u_i-\sum_{i=t+1}^k u_i, & i=t;\\
        0, & o/w.
    \end{cases}
    \end{aligned} 
\end{equation}
    This satisfies $\sum_{i=1}^s u'_i=\sum_{i=s+1}^k u'_i$. Since the arcs in between $\{v_1,...,v_s\}$ and $\{v_{s+1},...,v_k\}$ form a complete bipartite graph, there exists a perfect $u'$-matching $M$. For each $i\in\{1,...,s\}$ and $j\in\{s+1,...,k\}$, suppose $M$ has multiplicity $x_{ij}\in\Z_{\geq 0}$ on the arc $v_iv_j$. Let \begin{equation}\label{eq:new_weight}
    \begin{aligned}
        w'(v_iv_j)=\begin{cases}
        w(v_iv_j)+x_{ij}, & i\in\{1,...,s\},\ j\in\{s+1,...,k\};\\
         w(v_{i}v_j)+u_j-u_j', & i=s+1,\ j\in\{s+2,...,k\};\\
        w(v_iv_j), & o/w.
    \end{cases}
    \end{aligned} 
\end{equation}
An example is given in Figure \ref{fig:transfer}.
\begin{figure}[htbp]
	\centering
	\includegraphics[scale=0.26]{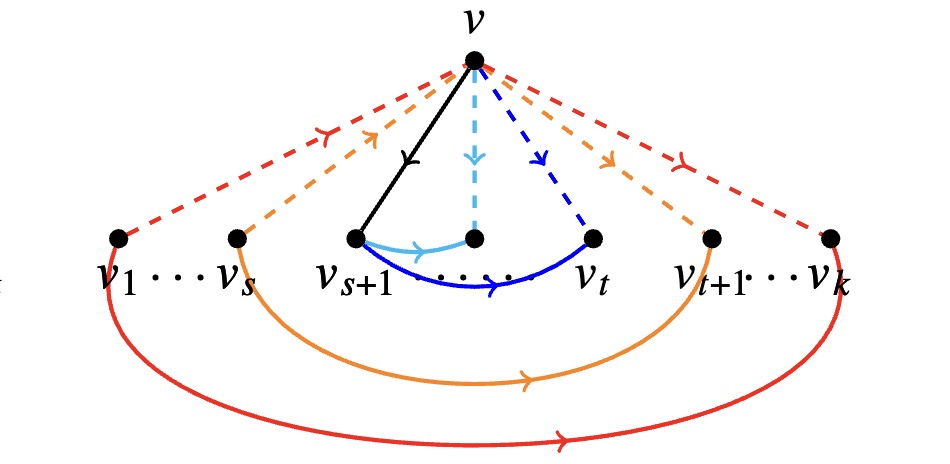}
  \caption{In the unweighted case, i.e., $u_1=u_2=\cdots=u_k=1$, one has $t=k-s+1$. Thus, $u'_1=\cdots=u'_s=1$, $u'_t=\cdots=u'_k=1$, and $u'_{s+1}=\cdots=u'_{t-1}=0$. The orange and red arcs form a perfect $u'$-matching. To go from $w$ to $w'$, the weights of all the colored solid arcs increase by $1$ and the weights of other arcs do not change.}
  \label{fig:transfer}
\end{figure} 

\paragraph{\textbf{Case 2}}: $\delta(v)$ is a dicut.

We may assume w.l.o.g. that $s=0$, which means $v$ is a source (the case $s = k$ corresponds to
the case $s = 0$ after flipping all arcs). In this case definition \eqref{eq:w'-matching} enforces $u'=0$ and definition \eqref{eq:new_weight} simplifies to
\[\begin{aligned}
        w'(v_iv_j)=\begin{cases}
         w(v_1v_j)+u_j, & j\in\{2,...,k\};\\
        w(v_iv_j), & o/w.
    \end{cases}
    \end{aligned}
\]

    \begin{theorem}\label{thm:weight_tau}
        The minimum weight of a dicut under $w'$ is at least $\tau$.
    \end{theorem}
    \begin{proof}
    For every dicut $\delta^+(U)$ in $D'$ that does not intersect $v_1\rightarrow v_2\rightarrow\ldots\rightarrow v_k$, either $\delta^+(U)$ or $\delta^+(U\cup \{v\})$ is also a dicut in $D$ with the same arc set, and thus its weight is at least $\tau$. Otherwise, assume $v_rv_{r+1}\in\delta^+(U)$ for some $r\in\{1,...k-1\}$. If $r\leq s$, then 
    \[
    \begin{aligned}
    w'(\delta^+(U))=&w(\delta^+(U))-\sum_{i=1}^r u_i+\sum_{i=1}^r\sum_{j=s+1}^k x_{ij}\\
    =&w(\delta^+(U))-\sum_{i=1}^r u_i+\sum_{i=1}^r u'_i=w(\delta^+(U))\geq \tau. 
    \end{aligned}
    \] 
    If $r\geq s+1$, then 
    \[
    \begin{aligned}
            w'(\delta^+(U))=&w(\delta^+(U\cup \{v\}))-\sum_{j=r+1}^k u_j+\sum_{i=1}^s\sum_{j=r+1}^k x_{ij}+\sum_{j=r+1}^k (u_j-u_j')\\
            =&w(\delta^+(U)\cup \{v\})-\sum_{j=r+1}^k u_j+\sum_{j=r+1}^k u'_j+\sum_{j=r+1}^k (u_j-u_j')\\
            =&w(\delta^+(U)\cup \{v\})\geq \tau.
    \end{aligned}
    \]
    Therefore, we proved that the minimum weight of a dicut in $D'$ is at least $\tau$.
\end{proof}

\subsection{Mapping the dijoins back}
    By the induction hypothesis, we can find a packing of $\tau$ dijoins in $(D',w')$. We prove that those dijoins can be used to construct a packing of $\tau$ dijoins in the original digraph $D$.

    \begin{theorem}\label{thm:map_back}
        If there exists a packing of $\tau$ dijoins in $(D',w')$, then there exists a packing of $\tau$ dijoins in $(D,w)$.
    \end{theorem}

    \begin{proof}
    Let $\mathcal{J}'=\{J'_1,...,J'_\tau\}$ be a packing of $\tau$ dijoins in $(D',w')$. We map each one to a dijoin of $D$ according to the following two cases.
    
    \paragraph{\textbf{Case 1}:} $\delta(v)$ is not a dicut (see Figure \ref{fig:weight_transfer1}).
    
    By Lemma \ref{lemma:dicut_preserve}, the family of dicuts in $D'$ is the same as the family of dicuts in $D$ restricted to $A'$. For each $i\in\{1,...,s\}$ and $j\in\{s+1,...,k\}$, suppose $y_{ij}$ dijoins in $\mathcal{J}'$ use $v_iv_j$. Let $\mathcal{J}'_{ij}\subseteq \mathcal{J}'$ be a family of dijoins with  $|\mathcal{J}'_{ij}|=\min\{x_{ij}, y_{ij}\}$ such that each $J'\in \mathcal{J}'_{ij}$ uses $v_iv_j$. Let $\mathcal{J}_{ij}:=\{J'-v_iv_j+v_iv+vv_j\mid J'\in \mathcal{J}'_{ij}\}$. Every $J\in \mathcal{J}_{ij}$ is a dijoin in $D$ because every dicut containing $v_iv_j$ also contains one of $v_iv$ and $vv_j$. Replace $\mathcal{J}'_{ij}$ with $\mathcal{J}_{ij}$ in $\mathcal{J}'$ and go to the next $(i,j)$. For each $j\in\{s+2,...,k\}$, suppose $z_j$ dijoins in $\mathcal{J}'$ use $v_{s+1}v_j$. Let $\mathcal{J}'_j\subseteq \mathcal{J}'$ be a family of dijoins with $|\mathcal{J}'_j|=\min\{z_j,u_j-u_j'\}$ such that each $J'\in \mathcal{J}'_j$ uses $v_{s+1}v_j$. Let $\mathcal{J}_j:=\{J'-v_{s+1}v_j+vv_j\mid J'\in \mathcal{J}'_j\}$. Every $J\in \mathcal{J}_j$ is a dijoin in $D$ because every dicut containing $v_{s+1}v_j$ also contains $vv_j$. Replace $\mathcal{J}'_j$ with $\mathcal{J}_j$ in $\mathcal{J}'$ and go to the next $j$.
    Let $\mathcal{J}$ be the new family of dijoins.
    By the way we construct $w'$,  $\mathcal{J}$ is indeed a valid packing in $(D,w)$.
\begin{figure}[htbp]
	\centering
	\includegraphics[scale=0.2]{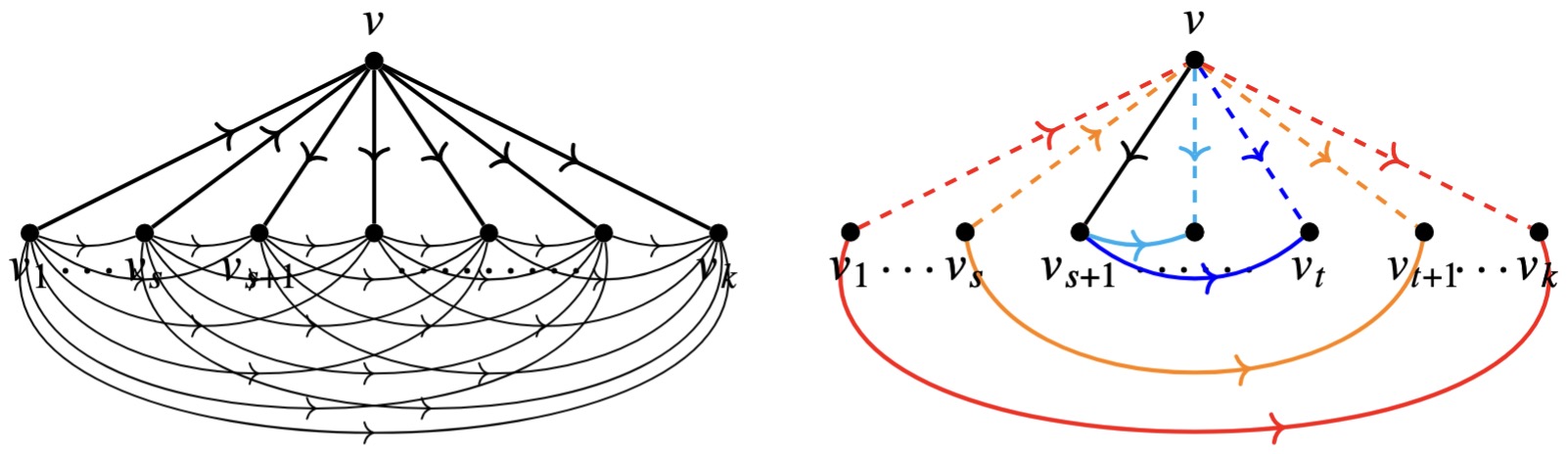}
  \caption{In the unweighted case, for each color class, one dijoin in $\mathcal{J}'$ using the solid arc is mapped to one dijoin using the dashed arc in $\mathcal{J}$ of the same color.}
  \label{fig:weight_transfer1}
\end{figure}

    \paragraph{\textbf{Case 2}:} $\delta(v)$ is a dicut (see Figure \ref{fig:weight_transfer2}). 
    
    It is further required that the dijoins cover $\delta(v)$. 
    Note that there is no $\mathcal{J}_{ij}$ in this case.
    Construct $\mathcal{J}_{j}'$ and $\mathcal{J}_{j}$ for $j\in\{2,...,k\}$ in the same way as in the previous case. Let $\mathcal{J}'_1:=\mathcal{J}'\backslash(\bigcup_{2\leq j\leq k} \mathcal{J}'_j)$
    be the remaining dijoins. To construct $\mathcal{J}$, we replace each $\mathcal{J}'_j$ by $\mathcal{J}_j$, $j\in\{2,...,k\}$. For each remaining dijoin $J'\in \mathcal{J}'_1$, we add one arc $vv_j$ to it for some $j\in \{1,...,k\}$. It suffices to argue that the number of remaining dijoins does not exceed the remaining total capacity of arcs incident to $v$.
    The crucial observation is that we can assume w.l.o.g. that every dijoin in $\mathcal{J}'$ uses at most one arc in $\{v_1v_j\mid j=2,...,k\}$. Indeed, if a dijoin $J'$ uses both $v_1v_i$ and $v_1v_j$ for some $2\leq i<j\leq k$, then $J'-v_1v_i$ is also a dijoin of $D'$ because every dicut containing $v_1v_i$ also contains $v_1v_j$. This implies that $\mathcal{J}_j'$'s, $j\in\{2,...,k\}$ are pairwise disjoint.
    Thus, $|\mathcal{J}_1'|=\tau-\sum_{j=2}^k |\mathcal{J}'_j|=\tau-\sum_{j=2}^k \min\{z_j, u_j\}\leq u_1+\sum_{j=2}^k (u_j-\min\{z_j, u_j\})$. The inequality follows from the fact that $\delta(v)$ is a dicut and thus has weight at least $\tau$. Since $u_1+\sum_{j=2}^k (u_j-\min\{z_j, u_j\})$ is the remaining total capacity of arcs incident to $v$, we have enough capacity for each remaining dijoin to include an arc incident to v. 
    
    \end{proof}

\begin{figure}[htbp]
	\centering
	\includegraphics[scale=0.2]{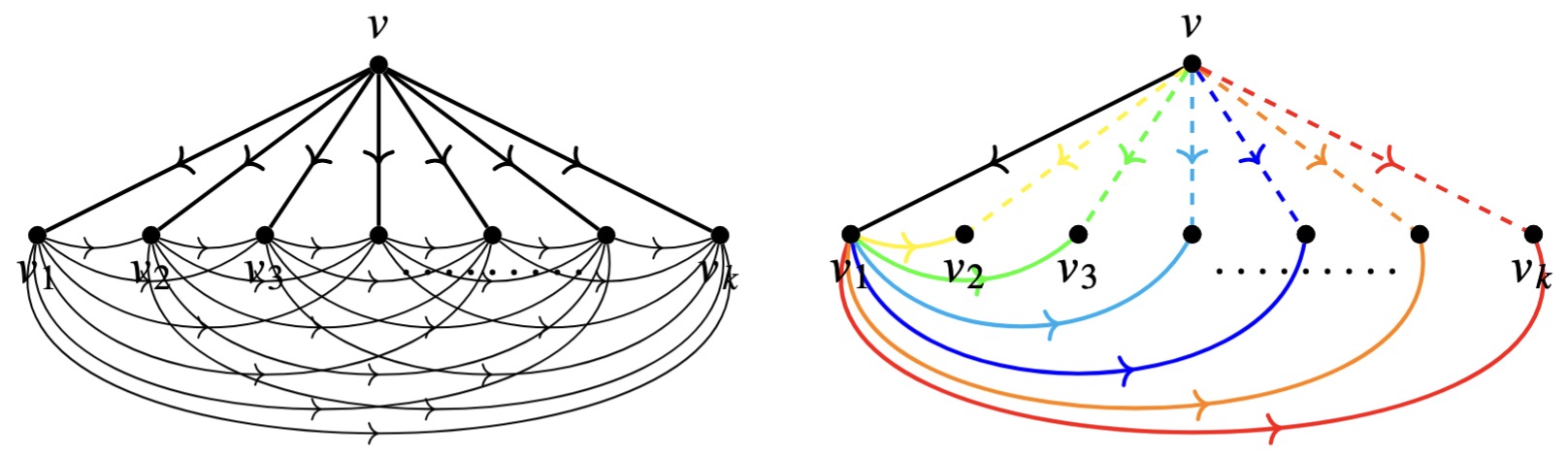}
  \caption{In the unweighted case, for each color class, one dijoin in $\mathcal{J}'$ using the solid arc is mapped to one dijoin using the dashed arc in $\mathcal{J}$ of the same color.}
  \label{fig:weight_transfer2}
\end{figure}

\section{Strongly polynomial-time algorithms}

In order to give a strongly polynomial-time algorithm, we need to store the packing in a compact way. Let $\chi_J$ be the incidence vector of a dijoin $J$. A packing of dijoins $\sum_{J\in \mathcal{J}} \lambda(J)\chi_{J}\leq w$ is stored as a family $\mathcal{J}$ of distinct dijoins and the multiplicity $\lambda(J)$ of each dijoin $J$ in the packing.
Our algorithm is given in Algorithm \ref{alg}.
\begin{algorithm}
		\caption{{\sc PackingDijoins}}
		\begin{algorithmic}[1]
			\Require A chordal digraph $D=(V,A)$ and weight $w: A\rightarrow \Z_{\geq 0}$ with minimum weight of a dicut at least $\tau$.
			\Ensure A family of dijoins $\mathcal{J}$ and their multiplicity $\lambda$ such that $\sum_{J\in \mathcal{J}}\lambda(J)=\tau$.
   \If{$|V|=1$}
   \Return{$\mathcal{J}=\{\emptyset\}$ and $\lambda(\emptyset)=\tau$}.
   \EndIf
                \State Find a simplicial vertex $v$ and let $D'$ be the digraph after deleting $v$;
                \State Compute $u'$ according to \eqref{eq:w'-matching};
                \State Construct a perfect $u'$-matching $M$ whose support is acyclic; 
                \State Compute $w'$ according to \eqref{eq:new_weight};
                \State $\mathcal{J}',\lambda'\gets$ \Call{PackingDijoins }{$D',w'$};
			\State $\mathcal{J},\lambda\gets$ \Call{MappingBack }{$\mathcal{J}',\lambda'$};
            \State \Return{$\mathcal{J},\lambda$}.
		\end{algorithmic}
  \label{alg}
	\end{algorithm}
The natural greedy way of mapping dijoins back from $D'$ to $D$, which is shown in Algorithm \ref{alg:support}, allows us to control the support of dijoins in the packing. The {\em support} of dijoins in the packing is the number of distinct dijoins $|\mathcal{J}|$. Let $n:=|V|$ and $m:=|A|$. We show that we can always build the packing of dijoins in a way that the support is bounded by $m-n+2$. To prove that this further implies that the algorithm runs in polynomial time, we first prove the following lemma:
\begin{algorithm}[htbp]
		\caption{{\sc MappingBack}}
		\begin{algorithmic}[1]
			\Require A packing of dijoins $(\mathcal{J}',\lambda')$ in $(D',w')$ such that $\sum_{J\in\mathcal{J}'}\lambda'(J)=\tau$.
			\Ensure A packing of dijoins $(\mathcal{J},\lambda)$ in $(D,w)$ such that $\sum_{J\in\mathcal{J}}\lambda(J)=\tau$.
                \State $(\mathcal{J},\lambda)\gets (\mathcal{J}',\lambda')$.
                \For{$e=v_iv_j$ with $w'(e)>w(e)$}
                \While{$w'(e)>w(e)$ and $\{J\in \mathcal{J}\mid e\in J\}\neq \emptyset$}
                \State Pick an arbitrary $J\in\mathcal{J}$ such that $e\in J$;
                \State $J'\gets J-v_iv_j+v_iv+vv_j$ if $i\leq s$ and $J'\gets J-v_iv_j+vv_j$ if $i=s+1$;
                \State $\lambda(J')\gets \min\{\lambda(J),w'(e)-w(e)\}$;
                \State $\lambda(J)\gets\lambda(J)-\lambda(J')$;
                \State $\mathcal{J}\gets \mathcal{J}+J'$;
                \State Delete $J$ from $\mathcal{J}$ if $\lambda(J)=0$;
                \State $w'(e)\gets w'(e)-\lambda(J')$.
                \EndWhile
            \EndFor
        \If{$\delta(v)$ is a dicut}
        \For{$e=vv_j$, $j\in\{2,...,k\}$ with $w(e)>\sum_{J\in\mathcal{J}, e\in J}\lambda(J)$}
                \While{$w(e)>\sum_{J\in\mathcal{J}, e\in J}\lambda(J)$ and $\big\{J\in \mathcal{J}\mid J\cap \delta(v)=\emptyset\big\}\neq \emptyset$}
                \State Pick an arbitrary $J\in\mathcal{J}$ such that $J\cap \delta(v)=\emptyset$;
                \State $J'\gets J+vv_j$;
                \State $\lambda(J')\gets \min\{\lambda(J),w(e)-\sum_{J\in\mathcal{J}, e\in J}\lambda(J)\}$;
                \State $\lambda(J)\gets\lambda(J)-\lambda(J')$;
                \State $\mathcal{J}\gets \mathcal{J}+J'$;
                \State Delete $J$ from $\mathcal{J}$ if $\lambda(J)=0$;
                \EndWhile
            \EndFor
        \For{$J\in \mathcal{J}$ with $J\cap \delta(v)=\emptyset$}
        \State $J'\gets J+vv_1$;
        \State $\lambda(J')\gets \lambda(J)$;
        \State $\mathcal{J}\gets \mathcal{J}+J'$;
        \State Delete $J$ from $\mathcal{J}$;
        \EndFor
    \EndIf
    \State \Return{$\mathcal{J},\lambda$}.
    \end{algorithmic}
  \label{alg:support}
	\end{algorithm}

\begin{lemma}\label{lemma:support}
    The support of dijoins in the packing constructed by Algorithm \ref{alg} is at most $m-n+2$, where $n=|V|$ and $m=|A|$.
\end{lemma}
\begin{proof}
     In the base step, when $|V|=1$, the packing is the empty set. In each iteration, we remove one simplicial vertex. When constructing $\mathcal{J}$ from $\mathcal{J}'$, $|\mathcal{J}|$ can potentially increase. We bound the number of new dijoins in $\mathcal{J}$ in each iteration, depending on whether $\delta(v)$ is a dicut. 
     
     If $\delta(v)$ is not a dicut, the new dijoins are the ones using arcs $e$ with weight $w'(e)>w(e)$. Those arcs are either in the perfect $u'$-matching $M$ or from $v_{s+1}$ to $\{v_{s+2},...,v_t\}$. Note that $M$ is actually induced on $\{v_1,v_2,...,v_s\}\cup \{v_t,v_{t+1},...,v_k\}$ since $u_j'=0$ for $j\in \{s+1,...,t-1\}$. From the standard dimension counting argument, we may assume that the support of $M$ is acyclic, and thus its size is at most $s+k-t$. Thus, there are at most $(s+k-t)+(t-s-1)=k-1$ arcs $e$ with $w'(e)>w(e)$. Observe that for each such arc $e$, only the last run of the first while loop of Algorithm \ref{alg:support} can increase $|\mathcal{J}|$. Indeed, whenever $|\mathcal{J}|$ increases, in line $9$ of the algorithm, $\lambda(J)\neq 0$, which means in line $6$, $\lambda(J')\gets w'(e)-w(e)$, after which the while loop terminates. Therefore, $|\mathcal{J}|$ increases by at most $k-1$ in this iteration. 

     If $\delta(v)$ is a dicut, each arc $e$ with $w'(e)>w(e)$ has the form $e=v_1v_j$ for some $j\in\{2,...,k\}$. Observe that for each such $e=v_1v_j$, if $|\mathcal{J}|$ increases during the first while loop of Algorithm \ref{alg:support}, then after its running $\sum_{J\in\mathcal{J}, vv_j\in J} \lambda(J)=w'(e)-w(e)=w(vv_j)$, in which case $|\mathcal{J}|$ will not increase during the second while loop when $e=vv_j$. A similar argument as in the previous case shows that for each arc $e=vv_j$, only the last run of the second while loop can increase $|\mathcal{J}|$. This shows that for each $j\in\{2,...,k\}$, $|\mathcal{J}|$ increases by at most $1$ during the two while loops when $e=v_1v_j$ and $e=vv_j$. The last for loop does not increase $|\mathcal{J}|$. Therefore, $|\mathcal{J}|$ increases by at most $k-1$ in this iteration as well.
     
     Since $k$ is the degree of $v$ in the current graph, there are at most $m-n+2$ distinct dijoins at the end of the algorithm.
\end{proof}

Now, we analyze the running time. The perfect elimination scheme can be computed in time $O(m+n)$  \cite{roseluekertarjan1976}. Suppose such a perfect elimination scheme is given. In each iteration, we construct a perfect $u'$-matching $M$ in a complete bipartite graph induced on $N(v)$. The greedy algorithm that repeatedly saturates the degree requirement of each vertex takes time $O(d(v))$, where $d(v)$ is the degree of $v$ in the current graph. For each call of Algorithm \ref{alg:support}, there are at most $O(d(v))$ for loops, and at most $O(|\mathcal{J}|)$ while loops. Therefore, it takes  $O(|\mathcal{J}|d(v))$ time to construct $\mathcal{J}$ from $\mathcal{J}'$. Together with the fact from Lemma \ref{lemma:support} that $|\mathcal{J}|\leq m-n+2$, Algorithm \ref{alg} runs in time $O(m|\mathcal{J}|+n)=O(m^2+n)$ in total.

\section{Conclusion and discussion}
We proved that the Edmonds-Giles conjecture is true for a digraph $D=(V,A)$ whose underlying undirected graph is chordal. Moreover, we may assume that the digraph is transitively closed, meaning that $uv,vw\in A$ implies $uw\in A$. This is because otherwise we can add arc $uw$ to $A$ and set its weight to $0$. This does not affect the structure and weights of the dicuts. Thus, we may assume that $D$ is the comparability digraph of a poset. Our proofs more generally imply that a minimal counterexample to the Edmonds-Giles conjecture does not contain a vertex whose neighbors form a chain in the poset. As a generalization of chordal graphs, a graph is \textit{$k$-chordal} if there is no chordless cycle of length more than $k$. We further conjecture that the Edmonds-Giles conjecture holds true if the underlying undirected graph is $5$-chordal. In a transitively closed digraph, every $(2k+1)$-chordal graph is also $2k$-chordal. Thus, it suffices to prove the conjecture for $4$-chordal graphs.  As Schrijver's counterexample has a chordless cycle of length $6$, this is the best possible.

\begin{credits}
\subsubsection{\ackname} 
We would like to thank two anonymous reviewers for their careful reading and constructive feedback on the paper. We also thank Ahmad Abdi and Olha Silina for valuable discussions on this topic. This work was supported in part by ONR grant N00014-22-1-2528, the Air Force Office of
Scientific Research under award number FA9550-23-1-0031, and a Balas PhD Fellowship from the Tepper School at Carnegie Mellon University.
\end{credits}

%
%
\newpage
\bibliographystyle{splncs04}
\bibliography{reference}
%





\end{document}